\documentclass[12pt,oneside]{article}
\usepackage[T1]{fontenc}
\usepackage{a4}
\usepackage{amsmath}
\usepackage{amsfonts}
\usepackage{color}
\usepackage{amssymb}
\usepackage{tikz}

\usepackage{verbatim}
\newtheorem{thm}{Theorem}
\newtheorem{lem}[thm]{Lemma}

\newtheorem{prop}[thm]{Proposition}
\newtheorem{wn}[thm]{Corollary}

\newtheorem{obs}[thm]{Observation}

\newenvironment{proof}{{\bf Proof.}}{\hspace*{\fill} \rule{2mm}{2mm} \par \hspace{0.1mm}}

\title{Domination subdivision and domination multisubdivision numbers of graphs}
\author{Magda Dettlaff$^{\dagger}$, Joanna Raczek$^{\dagger}$ and
 Jerzy Topp$^{\ddagger}$
\\
\\
$^{\dagger}${\small Faculty of Applied Physics and Mathematics } \\ {\small Gdansk University of Technology, ul. Narutowicza 11/12, 80-233 Gda\'nsk, Poland }\\ {\small {\tt mdettlaff\@@mif.pg.gda.pl}}, {\small {\tt Joanna.Raczek\@@pg.gda.pl}}
\\
$^{\ddagger}${\small Faculty of Mathematics, Physics and Informatics 
}\\
{\small  University of Gda\'nsk, ul. Wita Stwosza 57, 80-952 Gda\'nsk, Poland} \\{\small {\tt Jerzy.Topp\@@inf.ug.edu.pl}}}

\date{}

\begin{document}
\maketitle
\begin{abstract} The \emph{domination subdivision number} sd$(G)$ of a graph $G$ is the minimum number of edges that must be subdivided (where an edge can be subdivided at most once) in order to increase the domination number of $G$. It has been shown \cite{vel} that sd$(T)\leq 3$ for any tree $T$. We prove that the decision problem of the domination subdivision number is NP-complete even for bipartite graphs. For this reason we define the \emph{domination multisubdivision number} of a nonempty graph $G$ as a minimum positive integer $k$ such that there exists an edge which must be subdivided $k$ times to increase the domination number of $G$. We show that msd$(G)\leq 3$ for any graph $G$. The domination subdivision number and the domination multisubdivision numer of a graph are incomparable in general case, but we show that for trees these two parameters are equal. We also determine domination multisubdivision number for some classes of graphs. 
\end{abstract}



\section{Introduction and motivation}

For domination problems, multiple edges and loops are irrelevant, so we forbid them. Additionally, in this paper we consider connected graphs only. We use $V(G)$ and $E(G)$ for the vertex set and the edge set of a graph $G$ and denote $|V(G)|=n$, $|E(G)|=m$.  

A subset $D$ of $V(G)$ is \emph{dominating} in $G$ if every vertex of $V(G)-D$ has at least one neighbour in $D.$ Let $\gamma (G)$ be the minimum cardinality among all dominating sets in $G.$ A minimum dominating set of a graph $G$ is called a $\gamma (G)$-set.

For a graph $G=(V,E)$ subdivision of the edge $e=uv\in E$ with vertex $x$ leads to a graph with vertex set $V\cup \{x\}$ and edge set $(E-\{uv\})\cup \{ux,xv\}$. 
Let $G_{e,t}$ denote graph obtained from $G$ by subdivision of the edge $e$ with $t$ vertices (instead of edge $e=uv$ we put a path $(u,x_1,x_2,\ldots ,x_t,v)$). For $t=1$ we write $G_e$.


The \emph{domination subdivision number}, sd$(G)$, of a graph $G$ is the minimum number of edges which must be subdivided (where each edge can be subdivided at most once) in order to increase the domination number. Since the domination number of the graph $K_2$ does not increase when its only edge is subdivided, we consider subdivision number for connected graphs of order at least $3$. The domination subdivision number was defined by Velammal in 1997 (see \cite{vel}) and since then it is widely studied in graph theory papers.  This parameter was studied in trees by Aram, Sheikholeslami and by Favaron~\cite{fav1} and also by  Benecke and Mynhardt~\cite{myn}. General bounds and properties has been studied for example by Haynes, Hedetniemi and Hedetniemi~\cite{HHH}, by  Bhattacharya and  Vijayakumar~\cite{bhv}, by Favaron, Haynes and Hedetniemi~\cite{fav3} and  by  Favaron, Karami and Sheikholeslami~\cite{fav2}. In this paper we continue the study of domination subdivision numbers of graphs by proving that the decision problem of domination subdivision number is NP-complete even for bipartite graphs. For this reason we define msd$(uv)$ to be the minimum number of subdivisions of the edge $uv$ such that $\gamma(G)$ increases. Moreover, let the \emph{domination multisubdivision number} of a graph $G, m>0$, denoted by msd$(G)$, be defined as 
\[
\mbox{msd}(G)=\min \{\mbox{msd}(uv):\ uv\in E(G)\}.
\]

Domination multisubdivision number is well defined for all graphs having at least one edge.

\section{Notation}
The \emph{neighbourhood} $N_{G}(v)$ of a vertex $v\in V(G)$ is the set of all vertices adjacent to $v.$  The \emph{degree} of a vertex $v$ is $d_{G}(v)=|N_{G}(v)|.$ 

We say that a vertex $v$ of a graph $G$ is a {\it leaf\/} if $v$ has exactly one neighbour in $G$.  A vertex $v$ is called a {\it support vertex} if it is adjacent to a leaf. If $v$ is adjacent to more than one leaf, then we call $v$ a {\it strong support vertex}. 05C99

A path $(x,v_1,\ldots ,v_l,y)$ connecting two vertices $x$ and $y$ in a graph $G$ we call an $(x-y)$--\emph{path}. The vertices $v_1,\ldots ,v_l$ are its internal vertices. The length of a shortest such path is called the distance between $x$ and $y$ and denoted $d_G(x,y)$. The diameter diam$(G)$ of a connected graph G is the maximum distance between two vertices of $G$. For subsets $X$ and $Y$ of $V(G)$, an $(X-Y)$--\emph{path} is a path which starts at a vertex of $X$, ends at a vertex of $Y$, and whose internal vertices belong to neither $X$ nor $Y$. If $X=\{x\}$, then we write $(x-Y)$--path.

The \emph{private neighbourhood of a vertex u with respect to a set} $D \subseteq V (G)$, where $u \in D$, is the set PN$_G[u,D] = N_G[u] - N_G[D - \{u\}]$. If $v \in $PN$_G[u,D]$, then we say that $v$ is a private neighbour of $u$ with respect to the set $D$.

For any unexplained terms and symbols see \cite{fund}.
\section{NP-completeness of domination subdivision problem}
The decision problem of domination subdivision problem is in this paper stated as follows:\\
DOMINATION SUBDIVISION NUMBER (DSN)\\
INSTANCE: Graph $G=(V,E)$ and the domination number $\gamma(G)$.\\
QUESTION: Is sd$(G)>1?$

\begin{thm}\label{thm:NPC}
\emph{DOMINATION SUBDIVISION NUMBER} is \emph{NP}-complete even for bipartite graphs.
\end{thm}

\begin{proof}
The proof is by a transformation from 3-SAT, which was proven to be \emph{NP}-complete in \cite{GJ79}. The problem 3-SAT is the problem of determining if there exists an interpretation that satisfies a given Boolean formula. The formula in 3-SAT is given in conjunctive normal form, where each clause contains three literals. We assume that the formula contains the instance of any literal $u$ and its negation $\neg u$ (in the other case all clauses containing the literal $u$ are satisfied by the true assignment of $u$). 

Given an instance, the set of literals $U=\{u_1,u_2,\ldots,u_n\}$ and the set of clauses $C=\{c_1,c_2,\ldots,c_m\}$ of 3-SAT, we construct the following graph $G$. For each literal $u_i$ construct a gadget $G_i$ on 6 vertices, where $u_i$ and $\neg u_i$ are the leaves (however $u_i$ and $\neg u_i$ not necessarily are to be leaves in $G$), see Fig.~\ref{f1}. 

\begin{figure}[h]
\begin{center}
\begin{tikzpicture}

\shade[ball color=black] (-1.6,1.5) circle (3pt) ;
\shade[ball color=black] (-1.6,3) circle (3pt)node[left] {$u_i$};
\shade[ball color=black] (0,0) circle (3pt);
\shade[ball color=black] (1.6,1.5) circle (3pt);
\shade[ball color=black] (1.6,3) circle (3pt) node[left] {$\neg{u_i}$} ;
\shade[ball color=black] (0,-1) circle (3pt);

\draw (-1.6,3) -- (-1.6,1.5) -- (0,0) -- (1.6,1.5) -- (1.6,3);
\draw (-1.6,1.5) -- (0,-1) -- (1.6,1.5);
\end{tikzpicture}
\end{center}
\caption{A gadged $G_i$}\label{f1}
\end{figure}
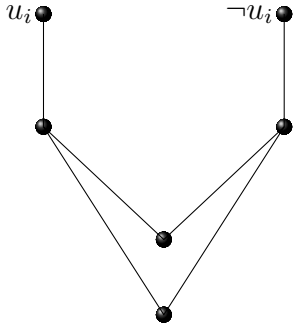

For each clause $c_j$ we have a clause vertex $c_j$, where vertex $c_j$ is adjacent to the literal vertices that correspond to the three literals it contains. For example, if $c_j=(u_1 \lor \neg u_2 \lor u_3 )$, then the clause vertex $c_j$ is adjacent to the literal vertices $u_1$, $\neg u_2$ and $u_3$. Then add new vertices $x_0,x_1$ in such a way that $x_1$ is adjacent to every clause vertex $c_j$ and to $x_0$. Hence $x_0$ is of degree one and $x_1$ is of degree $m+1$. Clearly we can see that $G$ is a bipartite graph and it can be build in polynomial time (see Fig.~\ref{f3}). 

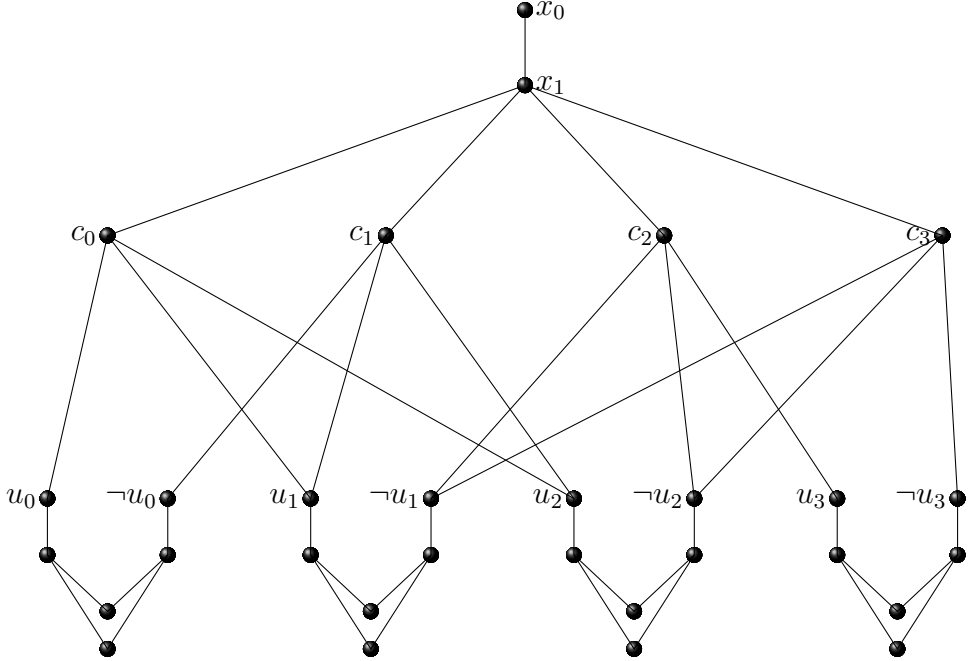
\begin{figure}[h]
\begin{center}
\begin{tikzpicture}

\foreach \n in {0,1,2,3} {\shade[ball color=black] (\n*3.5-0.8,0.75) circle (3pt)  ;
						\shade[ball color=black] (\n*3.5-0.8,1.5) circle (3pt) node[left] {$u_{\n}$};
						\shade[ball color=black] (\n*3.5,0) circle (3pt);
						\shade[ball color=black] (\n*3.5,-0.5) circle (3pt);
						\shade[ball color=black] (\n*3.5+0.8,0.75) circle (3pt);
						\shade[ball color=black] (\n*3.5+0.8,1.5) circle (3pt) node[left] {$\neg{u_{\n}}$} ;
						\draw (\n*3.5-0.8,1.5) -- (\n*3.5-0.8,0.75) -- (\n*3.5,0) -- (\n*3.5+0.8,0.75) -- (\n*3.5+0.8,1.5);
						\draw(\n*3.5-0.8,0.75) -- (\n*3.5,-0.5) -- (\n*3.5+0.8,0.75) ;
						};

\foreach \n in {0,1,2,3} {\shade[ball color=black] (\n*3.7,5) circle (3pt) node[left] {$c_{\n}$};
						\draw (\n*3.7,5) -- (5.55,7);}

\shade[ball color=black] (5.55,7) circle (3pt) node[right] {$x_1$};
\shade[ball color=black] (5.55,8) circle (3pt) node[right] {$x_0$};
\draw (5.55,7) -- (5.55,8);

\foreach \n in {0,1,2}{\draw (0,5)--(\n*3.5-0.8,1.5);};
\foreach \n in {1,2,3}{\draw (11.1,5)--(\n*3.5+0.8,1.5);};

\draw (0.8,1.5) -- (3.7,5) -- (2.7,1.5);
\draw (3.7,5) -- (6.2,1.5);
\draw (7.4,5) -- (4.3,1.5) ;
\draw (7.8,1.5) -- (7.4,5)--(9.7,1.5);
\end{tikzpicture}
\end{center}
\caption{A construction of $G$ for  $(u_0\lor u_1\lor u_2)\land (\neg u_0\lor u_1\lor u_2) \land(\neg u_1\lor \neg u_2\lor u_3) \land (\neg u_1\lor \neg u_2 \lor \neg u_3)$}\label{f3}
\end{figure}

First observe that at least two vertices from each gadget $G_i$ and either $x_1$ or $x_0$ must be contained in any minimum dominating set of $G$. Thus, $\gamma(G)\geq 2n+1$. On the other hand, is possible to construct o dominating set of $G$ of cardinality $2n+1$. Therefore, $\gamma(G)=2n+1$.

Denote by $G_1, G_2,\dots, G_{m(G)}$ the graph obtained from $G$ by subdividing once edge $e_1,e_2,\dots, e_{m(G)}$, respectively.
For a given graph $G$ and its domination number $\gamma(G)$ it is possible to verify a certificate for the DSN problem, which are dominating sets of cardinality $\gamma(G)$ in $G_1, G_2,\dots, G_{m(G)}$, in polynomial time.

Assume first $C$ has a satisfying truth assignment. If we subdivide any edge belonging to a gadged $G_i$, then we may construct a minimum dominating set of the resulting graph by adding to it two vertices from each gadged $G_i$ and additionally $x_1$. The situation is similar if we subdivide any edge incident with a clause vertex. Now let $x$ be the new vertex obtained by subdivision the edge $x_0x_1$ in $G$ and denote by $G_x$ the obtained graph. Since $C$ has a satisfying truth assignment, the minimum dominating set of $G_x$ is constructed by taking the vertices defined by the truth assignment together with one more vertex from each gadged $G_i$ and together with $x$. Therefore we conclude that $sd(G)>1$.

Assume now $C$ does not have a satisfying truth assignment. Then subdivide the edge $x_0x_1$ to obtain the graph $G_x$. The minimum dominating set of $G_x$ must contain at least two vertices from each gadged $G_i$ and additionally $x$. However, since $C$ does not have a satisfying truth assignment, no subset of $2n$ vertices of $G_1\cup G_2\cup\dots\cup G_n$ can dominate each gadged vertex and each cause vertex. Therefore, $sd(G)=1$.
\end{proof}

The decision problem of domination multisubdivision problem may be stated similarly:\\
DOMINATION MULTISUBDIVISION NUMBER (DMN)\\
INSTANCE: Graph $G=(V,E)$ and the domination number $\gamma(G)$.\\
QUESTION: Is msd$(G)>1?$

Let us observe, that
\begin{obs}\label{o1}
Let $G$ be a graph. Then
\[
sd(G)=1 \textrm{ if and only if } msd(G)=1.
\]
\end{obs}

This observation implies that the following result one may prove in similar manner as Theorem~\ref{thm:NPC}.

\begin{thm}\label{thm:NPC2}
\emph{DOMINATION MULTISUBDIVISION NUMBER} is \emph{NP}-complete even for bipartite graphs.  \hspace*{\fill} \rule{2mm}{2mm} \par \hspace{0.1mm}
\end{thm}

\section{Results and bounds for domination multisubdivision number}
Determining the domination multisubdivision number is hard even for bipartite graphs, so it is good motivation to study this parameter and give some general bounds and properties. Here we start with some basic properties of multisubdivision numbers.

From Observation~\ref{o1} and properties of graphs in which the subdivision number is one, follow the next two observations.
\begin{obs}\label{wn1}
If a graph $G$ has a strong support vertex, then 
\[
sd(G)=msd(G)=1.
\]
\end{obs}
\begin{obs}
For a complete graph $K_n$ and a wheel $W_n$, $n\geq 3$, we have
\[msd(K_n)=sd(K_n)=msd(W_n)=sd(W_n)=1.\]
\end{obs}

Since any cycle (any path) with subdivided edge $k$ times is isomorphic to the cycle (the path) with subdivided $k$ edges once, we immediately obtain the observation.

\begin{obs}
For a cycle $C_n$ and a path $P_n$, $n\geq 3$, we have
\begin{displaymath}
msd(C_n)=sd(C_n)=\left\{ \begin{array}{ll}
1 & \textrm{if $n\equiv 0\pmod 3$}\\
2 & \textrm{if $n\equiv 2\pmod 3$}\\
3 & \textrm{if $n\equiv 1\pmod 3,$}
\end{array} \right.
\end{displaymath}
\begin{displaymath}
msd(P_n)=sd(P_n)=\left\{ \begin{array}{ll}
1 & \textrm{if $n\equiv 0\pmod 3$}\\
2 & \textrm{if $n\equiv 2\pmod 3$}\\
3 & \textrm{if $n\equiv 1\pmod 3.$}
\end{array} \right.
\end{displaymath}
\end{obs}

\begin{thm}\label{obs1}
For a connected graph $G$, 
\[
1\leq msd(G)\leq 3.
\]
\end{thm}
\begin{proof}
Let $uv$ be an edge of a graph $G$. Since $\gamma(G_{uv})\geq \gamma (G)$, we have \mbox{msd$(G)\geq 1$.} Now, let us subdivide an edge $uv$ with $3$ vertices $x$, $y$ and $z$ (we replace the edge $uv$ with the path $(u,x,y,z,v)$), and let $D$ be a $\gamma (G_{uv,3})$-set. Since $D$ is a minimum dominating set, it is easy to observe that $1\leq |D\cap \{x,y,z\}|\leq 2$. It is again easy to observe that if $|D\cap \{x,y,z\}|=2$, then we can exchange one vertex from $D\cap \{x,y,z\}$ with $u$ or $v$ to obtain minimum dominating set of $G_{uv,3}$ such that $|D\cap \{x,y,z\}|=1$. Thus, if $x\in D$, then $v$ belongs to $D$ to dominate $z$ and $D-\{x\}$ is a $\gamma (G)$-set. Similarly, if $z\in D$, then $u\in D$ and $D-\{z\}$ a $\gamma (G)$-set. If $y\in D$, then obviously $D-\{y\}$ is a $\gamma (G)$-set. In all the cases we can find a smaller dominating set in $G$ then in $G_{uv,3}$, it implies that msd$(G)\leq 3$.
\end{proof}

\begin{prop}
For a complete bipartite graph $K_{p,q}$, $p\leq q$, we have
\begin{displaymath}
msd(K_{p,q})=\left\{ \begin{array}{ll}
1 & \textrm{if $p=1$ and $q>1$,}\\
2 & \textrm{if $p=q=1$,}\\
3 & \textrm{if $p\geq 2.$}
\end{array} \right.
\end{displaymath}
\end{prop}
\begin{proof}
The result is obvious for $p=1$. Thus, we assume that $2\leq p\leq q$.

Let $uv$ be any edge of $K_{p,q}$. Then $\{u,v\}$ is a minimum dominating set of the graph $K_{p,q}$ and of the graph $K_{p,q}$ with the edge $uv$ subdivided two times. This implies that msd$(K_{p,q})>2$ and therefore, by Theorem~\ref{obs1}, msd$(G)=3$.
\end{proof}

Although the multisubdivision number of a graph is bounded from above by~3, it was proven by Favaron, Karami and Sheikholeslami~{\cite{fav2}} that the subdivision number can be arbitrary large: {\it For each pair of positive integers r and q such that $r+q\geq 4$, there exists a graph G with $\delta(G)=r$ and sd$(G)\geq r+q$.} Hence, the difference between sd$(G)$ and msd$(G)$ also cannot be bounded from above by any integer in general case. Although the multisubdivision number is always not greater than three and the subdivision number cannot be bounded from above by any integer, the inequality msd$(G)\leq$ sd$(G)$ is not true, since msd$(K_{p,q})=3$ and sd$(K_{p,q})=2$ for $3\leq p\leq q$. Thus, the subdivision number and the multisubdivision number are incomparable in general case. In the next section we show that for trees these two domination parameters are the same.

\section{Domination multisubdivision number of a tree}
Now we consider multisubdivision numbers for trees. The main result of this section is what follows.
\begin{thm}\label{tg}
Let $T$ be a tree with $n\geq 3$. Then
\[
sd(T)=msd(T).
\]
\end{thm}
Thus, in trees it does not matter if we subdivide a set of edges, each edge once, or if we multi-subdivide only one edge. In both cases the minimum number of subdivision vertices needed to increase the domination number is the same for a tree.

It has been shown by Velammal \cite{vel} that the domination subdivision number of a tree is either 1, 2 or 3. The classes of trees $T$ with sd$(T)=1$ and sd$(T)=3$ are characterized (see \cite{fav1}, \cite{myn}). Thus by Observation~\ref{o1}, in order to prove Theorem~\ref{tg} it suffices to show that for a tree $T$ with at least~3 vertices we have $$sd(T)=3 \textrm{ if and only if } msd(T)=3.$$

\subsection{Trees with the domination multisubdivision number equal to 3}
The following constructive characterization of the family $\mathcal{F}$ of labeled trees $T$ with sd$(T)=3$ was given by Aram, Sheikholeslami and Favaron~\cite{fav1}. The label of a vertex $v$ is also called a status of $v$ and is denoted by sta$(v)$.
 Let $\mathcal{F}$ be the family of labelled trees such that:
 \begin{itemize}
\item contains $P_4$ where the two leaves have status $A$ and the two support vertices have status $B$; and
\item is closed under the two operations  $\mathcal{T}_1$ and  $\mathcal{T}_2$, which extend the tree $T$ by attaching a path to a vertex $v \in V (T )$.
\begin{enumerate}
\item [Operation  $\mathcal{T}_1$.] Assume $sta(v) = A.$ Then add a path $(x,y,z)$ and the edge $vx.$ Let $sta(x)=sta(y) = B$ and $sta(z) = A.$
\item [Operation  $\mathcal{T}_2$.] Assume $sta(v) =B.$ Then add a path $(x,y)$ and the edge $vx.$ Let $sta(x)=B$ and $sta(y)=A.$
\end{enumerate}
\end{itemize}

If $T\in \mathcal{F}$, we let $A(T)$ and $B(T)$ be the set of vertices of statuses $A$ and $B$, respectively, in $T$.


\begin{thm}\textnormal{\cite{fav1}}
For a tree T of order $n\geq 3$, 
\[
sd(T)=3 \textrm{ if and only if } T\in \mathcal{F}.
\]
\label{tw2}\hspace*{\fill} \rule{2mm}{2mm}
\end{thm}

In order to prove Theorem~\ref{tg}, we will need the following Observation~\ref{obs2} and Lemma~\ref{lem1} made for trees belonging to the family $\mathcal{F}$.

\begin{obs}\textnormal{\cite{fav1}}
Let $T\in \mathcal{F}$ and $v\in V(T)$.
\begin{itemize}
\item[$(1)$] If $v$ is a leaf, then sta$(v)=A$. 
\item[$(2)$] If $v$ is a support vertex, then sta$(v)=B$.
\item[$(3)$] If sta$(v)=A$, then $N(v)\subseteq B(T)$.
\item[$(4)$] If sta$(v)=B$, then $v$ is adjacent to exactly one vertex of $A(T)$ and at least one vertex of $B(T)$.
\item[$(5)$] The distance between any two vertices in $A(T)$ is at least $3$.
\end{itemize}
\label{obs2}
\end{obs}

\begin{lem}\textnormal{\cite{fav1}}
If $T\in \mathcal{F}$, then $A(T)$ is a $\gamma (T)$-set.
\label{lem1}\hspace*{\fill} \rule{2mm}{2mm}
\end{lem}

\begin{lem}
If $T$ is a tree with $sd(T)=3$, then $msd(T)=3$.
\label{t1}
\end{lem}
\begin{proof}
Let $T$ be a tree with sd$(T)=3$. Thus, by Theorem \ref{tw2}, $T\in \mathcal{F}$ and by Lemma \ref{lem1}, $A(T)$ is a $\gamma (T)$-set. 

By Theorem~\ref{obs1}, in order to prove the statement, it is enough to show that $msd(T)>2$. 

Let $uv\in E(T)$ be any edge. Then by Observation~\ref{obs2}, two cases are possible: either $\{sta(u), sta(v)\}=\{B\}$ or $\{sta(u), sta(v)\}=\{A, B\}$. We subdivide $uv$ with two vertices $x$ and $y$. Now we construct a minimum dominating set $D$ of $T_{uv,2}$ in a following way: we start with $A(T)$ and every vertex $a\in A(T)$ we replace with a vertex $a'$ which belongs to $(a-\{x,y\})$--path. If sta$(u)=$sta$(v)=B$, then $\{u,v\}\subset D$. If sta$(u)=A$ and sta$(v)=B$, then $\{x,v\}\subset D$. By Observation~\ref{obs2}, it is easy to check that $D$ is a dominating set of $T_{uv,2}$ and that $|D|=|A(T)|$. Since subdivision of the edge can not decrease the domination number of a graph, $D$ is a $\gamma(T_{uv,2})$--set. Hence, $\gamma(T)=\gamma(T_{uv,2})$, what implies msd$(T)=3$.
\end{proof}

\begin{lem}
If T is a tree with msd$(T)=3$, then sd$(T)=3$.
\label{t2}
\end{lem}
\begin{proof}
Let $T$ be a tree with msd$(T)=3$. By Theorem \ref{tw2}, it is enough to show that $T\in \mathcal {F}$. We consider trees with diam$(T)\geq 3$ (because for trees with diam$(T)\leq 2$ we have msd$(T)\leq 2$). Moreover, it is no problem to check that the result is true for all trees with at most~4 vertices: the only tree $T$ with msd$(T)=3$ and with at most~4 vertices is $P_4$ which belongs to $\mathcal {F}$. We continue the proof by induction on $n$, number of vertices of $T$. Assume that every tree $T'$ with $n'<n$ vertices such that msd$(T') = 3$ belongs to the family~$\mathcal{F}$.

Now, let $T$ be a tree with msd$(T)=3$, diam$(T)\geq 3$ and $n>4$. Then $\gamma (T)=\gamma(T_{e,2})$ for every edge $e\in E(T)$. Let $P=(v_0,v_1,v_2,\ldots ,v_k)$ be a longest path such that the degree of a vertex $v_2$ is as big as possible. It follows by Observation~\ref{wn1} that $d(v_1)=2$ (as otherwise $v_1$ is a strong support vertex and then msd$(T)=1$). Now we consider cases:
\begin{itemize}
\item[]\emph{Case 1. }$d(v_2)=2$.\\
Since msd$(T)=3$, $v_3$ is neither a support vertex nor a neighbor of a support vertex (as otherwise $\gamma(T_{v_1v_2,2})>\gamma (T)$). Thus, outside the path $P$, only $P_3$'s may be attached to $v_3$. We consider the tree $T'=T-\{v_0,v_1,v_2\}$. It is no problem to see that $\gamma(T)=\gamma (T')+1$. Moreover, for every edge $e\in E(T')$ we have $\gamma(T'_{e,2})=\gamma(T_{e,2})-1=\gamma(T)-1=\gamma(T')$. 
Hence, msd$(T')=3$ and from induction hypothesis $T'\in \mathcal{F}$. From the construction of a family $\mathcal {F}$ we know sta$(v_3)=A$. Thus $T$ can be obtained from $T'$ by Operation $\mathcal{T}_1$, where sta$(v_2)=$sta$(v_1)=B$ and sta$(v_0)=A$.

\item[]\emph{Case 2. }$d(v_2)>2$ and $v_2$ is a support vertex, say $v'_2$ is the leaf adjacent to $v_2$. \\
By Observation~\ref{wn1}, $v_2$ is adjacent to only one leaf. We consider the tree $T'=T-\{v_0,v_1\}$. It is obvious that $\gamma (T)=\gamma (T')+1$. Since msd$(T)=3$ and $v_1,v_2$ are support vertices, we have $\gamma(T'_{e,2})=\gamma(T_{e,2})-1=\gamma(T)-1=\gamma(T')$ for every edge $e\in E(T')-\{v_2v'_2\}$. This also implies that there exists a $\gamma (T')$-set $D'$ containing $v_2$ and $v_3$. We subdivide the edge $v'_2v_2$ with vertices $x$ and $y$. Then $(D'-\{v_2\})\cup \{x\}$ is a $\gamma$-set in $T'_{v_2v'_2,2}$ and $\gamma(T')=\gamma(T'_{v_2v'_2,2})$. Therefore $T'\in \mathcal{F}$ with sta$(v_2)=B$, and $T$ can be obtained from $T'$ by Operation $\mathcal{T}_2$, where sta$(v_1)=B$ and sta$(v_0)=A$.

\item[] \emph{Case 3. }$d(v_2)>2$ and $v_2$ is not a support vertex. \\
Then $v_2$ is adjacent to at least two support vertices. Let $T'=T-\{v_0,v_1\}$. Again $\gamma (T)=\gamma (T')+1$. Since msd$(T)=3$, there exist a minimum dominating set which contains $v_2$. Therefore for every edge $e\in E(T')$ we obtain $\gamma(T'_{e,2})=\gamma(T_{e,2})-1=\gamma(T)-1=\gamma(T')$. Hence, $T'\in \mathcal{F}$, sta$(v_2)=B$ and $T$ can be obtained from $T'$ by Operation $\mathcal{T}_2$, where sta$(v_1)=B$ and sta$(v_0)=A$.
\end{itemize}

In all these cases $T\in \mathcal{F}$.
\end{proof}

Now, Theorem~\ref{tg} is an immediate consequence of Lemmas~\ref{t1}, \ref{t2} and Observation~\ref{o1}.

\subsection{Trees with the domination multisubdivision number equal to 1}
In this subsection we shortly present a characterization of all trees $T$ with msd$(T)=1$. This characterization is an immediate consequence of Observation~\ref{o1} and results of Benecke and Mynhardt in~\cite{myn}, where they have characterized all trees with domination subdivision number equal to~1. Let $\mathcal{N}(G)$ consists of those vertices which are not contained in any $\gamma (G)$-set.
\begin{wn}\label{tt1}
For a tree T of order $n\geq 3$, $msd(T)=1$ if and only if $T$ has 
\begin{itemize}
\item[$i)$] a leaf $u\in \mathcal{N}(T)$ or
\item[$ii)$] an edge $xy$ with $x,y\in \mathcal{N}(T)$.
\end{itemize}
\end{wn}

\end{document}